\documentclass[12pt]{article}
\usepackage{amsmath,amsopn,amssymb,amsthm,amsfonts}
\usepackage[T2A]{fontenc}
\usepackage[cp1251]{inputenc}
\usepackage{longtable}
\usepackage{graphicx}
\graphicspath{{noiseimages/}}

\newtheorem{theorem}{Theorem}[section]

\newcommand{\w}{\omega}
\newcommand{\Z}{\mathbb{Z}}
\newcommand{\R}{\mathbb{R}}
\newcommand{\CC}{\mathbb{C}}

\begin{document}

{\bf \large
\centerline{N.~K.~Smolentsev, P.~N.~Podkur}

\vspace{3mm}
\centerline{Meyer wavelets with scaling factor $N>2$}
\vspace{3mm}
}
\begin{abstract}
In this paper, Meyer wavelets with an arbitrary integer scaling factor $N>2$ are defined using wavelets with multiple scaling factors $MN>2$. Expressions for frequency functions of wavelets and corresponding filters are obtained.
\end{abstract}

\section{Introduction} \label{Intro}

Wavelet analysis provides powerful tools for solving various problems of mathematical modeling. Wavelet analysis is used especially effectively to study various medical data \cite{Pavl}, \cite{Po-Sm-16}, \cite{Rahman}.
At present, many different wavelets are known and effectively used \cite{Daub}, \cite{Sm-14}. Among them, Meyer wavelets stand out in that they provide the clearest separation of the signal over frequency bands and have rapidly decreasing filters.
The use of powers of two for constructing the theory of wavelets and when using them is convenient in many respects, although it is not obligatory. Instead of the scaling factor 2, you can use any positive integer $N$ and actually a rational number greater than one \cite{Daub}. See \cite{Po-Sm-07} and \cite{Sm-14} for a development of this topic.
Wavelet analysis with a scaling factor $N>2$ has certain advantages, providing the separation of the signal into $N$ frequency bands even with a single wavelet decomposition.
This, for example, is relevant when analyzing the signals of cardiograms obtained using high-resolution electrocardiographs (15-20 kHz).
Examples of wavelet analysis of stock market data with scaling factors of 4, 5, and 8 are presented in \cite{Sm-14}. The work \cite{Po-Sm-20} shows the possibility of using wavelet analysis with a scaling factor of 3 to study EEG signals.

In this paper, Meyer wavelets with an arbitrary integer scaling factor $N>2$ are constructed. A general definition of Meyer wavelets with an arbitrary integer scaling factor $N>2$ is given.
The presented construction works well for odd $N>2$.
However, in the case of even N, one of the wavelet filters, namely the wavelet filter $\psi^{N-1}(x)$, has an infinite impulse response and a weak decay at infinity. To solve this problem, multiple scaling factors $MN$ are used in the work.
It is shown that if wavelets with scaling factors M and N are given, then wavelets with scaling factor $MN$ can be determined. It is also shown that if the given M- and N-wavelets have the Meyer function $\varphi(x)$ as a scaling function, then the result is $MN$-wavelets with the same Meyer scaling function $\varphi(x)$.
Considering that for $N=2$ the Meyer wavelets $\varphi(x)$ and $\psi(x)$ are well known and have rapidly decaying filters, this solves the above problem with wavelet filters for an even scaling factor and completely solves the problem of defining Meyer wavelets with an arbitrary integer scaling factor $N>2$. At the end of the paper, we consider an example for $M=3$ and $N=2$.

\section{Preliminaries} \label{Preface}
Let us recall the basic concepts of wavelet analysis with a integer scaling factor $N\ge 2$ (for more details, see \cite{Sm-14}). A function $\varphi(x)\in L^2(\R)$ is called scaling-function with a scale factor $N$ if it satisfies the relation
\begin{equation}\label{eq-1}
\varphi(x)=\sqrt{N}\sum_{n\in \Z}h_n^0 \varphi(Nx-n),
\end{equation}
where the set of real numbers $\{h_n^0\}$ is called scaling filter of the function $\varphi(x)$.
For the scaling function $\varphi(x)$, frequency function $H^0(\w)$ is determined by the formula:
\begin{equation}\label{eq-2}
H^0(\w)=\frac {1}{\sqrt{N}}\sum_{n\in \Z}h_n^0 e^{-in\w}.		
\end{equation}

After the Fourier transform, the scaling relation (\ref{eq-1}) takes the form
\begin{equation}\label{eq-3}
\widehat{\varphi}(\w)=H^0\left(\frac{\w}{N}\right)\widehat{\varphi}\left(\frac{\w}{N}\right).		
\end{equation}

For the scaling function $\varphi(x)$, there are $N-1$ wavelet functions $\psi^1(x), \dots , \psi^{N-1}(x)$ defined by the equalities
\begin{equation}\label{eq-4}
\psi^{k}(x)=\sqrt{N}\sum_{n\in \Z}h_n^k \varphi(Nx-n), \quad  k=1,2,\dots,N-1.		
\end{equation}
where the coefficients $\{h_n^k\}_{n\in \Z}$ are called wavelet filters.
In the frequency domain, the last relations (\ref{eq-4}) take the form:
\begin{equation}\label{eq-5}
\widehat{\psi^k}(\w)=H^k\left(\frac{\w}{N}\right)\widehat{\varphi}\left(\frac{\w}{N}\right), \quad  k=1,2,\dots,N-1,		
\end{equation}
where $H^k(\w)$ are frequency functions corresponding to wavelets $\psi^k(x)$:
\begin{equation}\label{eq-6}
H^k(\w)=\frac {1}{\sqrt{N}}\sum_{n\in \Z}h_n^k e^{-in\w}, \quad  k=1,2,\dots,N-1.		
\end{equation}

The wavelets are called orthogonal if the functions
$\psi_{j,n}^k(x)=\sqrt{N^j}\psi^k(N^j x-n)$, $n\in \Z$, $k = 1,2, \dots, N-1$ form a complete orthonormal system of functions in $L^2(\R)$.
For orthogonality, the frequency functions of wavelets must satisfy the unitarity property of matrix  \cite{Daub}, \cite{Sm-14}:
\begin{equation}\label{eq-7}
\left(\begin{matrix}
H^0(\w)&H^0\left(\w-\frac{2\pi}{N}\right)&\dots&H^0\left(\w-\frac{2\pi(N-1)}{N}\right)\\
H^1(\w)&H^1\left(\w-\frac{2\pi}{N}\right)&\dots&H^1\left(\w-\frac{2\pi(N-1)}{N}\right)\\
\dots&\dots&\dots&\dots\\
H^{N-1}(\w)&H^{N-1}\left(\w-\frac{2\pi}{N}\right)&\dots&H^{N-1}\left(\w-\frac{2\pi(N-1)}{N}\right)\\
\end{matrix}\right)
\end{equation}

In practice, in wavelet analysis, we are dealing with a digital signal, which is represented by an array $X = \{x_n\}$.
Then its wavelet decomposition is a decomposition into $N$ arrays: $X\to \{A^0, A^1, A^2, \dots, A^{N-1}\}$ and is produced according to the formulas \cite{Sm-14}:
$$
a_m^0=\sum_{n\in \Z}\overline{h_n^0}x_{n+Nm}, \quad a_m^k=\sum_{n\in \Z}\overline{h_n^k}x_{n+Nm}, \quad k = 1,2, \dots, N-1
$$
where the overline means complex conjugation.
The array $A^0=\{a_m^0\}$ is called the approximation coefficients of the signal $X=\{x_n\}$, and the remaining arrays $A^k=\{a_m^k\}$ are called the refinement coefficients.
The decomposition procedure can be repeated by applying it to the array of coefficients $A^0$.
The signal $X=\{x_n\}$ is reconstructed from the wavelet expansion coefficients as follows:
$$
x_n=\sum_{m\in \Z}h_{n-Nm}^0 a^0_{m}+ \sum_{k=1}^{N-1}\sum_{m\in \Z} h_{n-Nm}^k a^k_{m}.
$$

\section{Meyer wavelets} \label{Meyer_wavelets}
The Meyer scaling function $\varphi(x)$ for the scaling factor $N=2$ is defined \cite{Daub}, \cite{Sm-14} by specifying its Fourier transform by the equality
\begin{equation}\label{eq-8}
\widehat{\varphi}(\w)=
\left\{
  \begin{matrix}
    1, &\w\in [-\frac{2\pi}{3},\frac{2\pi}{3}], \\
    \cos\left(\frac \pi2 \nu\left(\frac{3}{2\pi}|\w|-1\right)\right), & \frac{2\pi}{3} \le|\w|\le \frac{4\pi}{3}, \\
    0, & \rm{for\ others}\ \w ,\\
  \end{matrix}
\right .
\end{equation}
where $\nu(x)$ is an auxiliary function that satisfies three conditions:
$\nu(x)= 0$ for $x\le 0$, $\nu(x)=1$ for $x\ge 0$, and $\nu(x)+\nu(1-x)=1$.
A common choice of the function $\nu(x)$ is the following polynomial interpolation
$\nu(x)= x^4(35- 84x+70x^2 -20x^3)$ between $0$ and $1$ on the interval $[0,1]$.

\begin{theorem} \label{Th1}
The Meyer scaling function $\varphi(x)$ is an $N$-scaling function for any integer number $N>2$.
\end{theorem}

\begin{proof}
It suffices to show that there exists a frequency function $H^0(\w)$ for which the scaling relation in the form of Fourier transforms is satisfied: $\widehat{\varphi}(\w) = H^0\left(\frac{\w}{N}\right)\widehat{\varphi}\left(\frac{\w}{N}\right)$.	
Due to $2\pi$-periodicity of function $H^0(\w)$, it suffices to define the $H^0(\w)$ on the interval $[-\pi, \pi]$.
Since the function $\widehat{\varphi}(\w)$ vanishes outside the interval $[-4\pi/3, 4\pi/3]$ and $\widehat{\varphi}\left(\frac{\w}{N}\right)=1$ on the interval $[-2\pi N/3,2\pi N/3]$, which includes $[-4\pi/3, 4\pi/3]$, then $H^0\left(\frac{\w}{N}\right) =\widehat{\varphi}(\w)$ on $[-4\pi/3, 4\pi/3]$, so $H^0(\w) =\widehat{\varphi}(N\w)$ on the interval $[-4\pi/3N, 4\pi/3N]$.
Thus, we obtain the required function $H^0(\w)$:
\begin{equation}\label{eq-9}
H^0(\w)=
\left\{
  \begin{matrix}
    1, &\w\in [-\frac{2\pi}{3N},\frac{2\pi}{3N}], \\
    \cos\left(\frac \pi2 \nu\left(\frac{3N}{2\pi}|\w|-1\right)\right), & \frac{2\pi}{3N} \le|\w|\le \frac{4\pi}{3N}, \\
    0, & \rm{for\ others}\ \w\in[-\pi,\pi] ,\\
  \end{matrix}
\right .
\end{equation}
Outside the interval $[-\pi,\pi]$ the $H^0(\omega)$ continues as $2\pi$-periodic function.
The coefficients $\{h_n^0\}$ of the $N$-scaling relation (\ref{eq-1}) are found from the Fourier expansion of the function $H^0(\w)$.
\end{proof}

\subsection{Construction of frequency functions and wavelets}
To construct the Meyer wavelets, we first find their frequency functions $H^k(\w)$ taking into account the unitarity of matrix (\ref{eq-7}).
Due to $2\pi$-periodicity of $H^k(\w)$, it suffices to define the function $H^k(\w)$ on the interval $[-\pi,\pi]$.
We first define $H^k(\w)$ on $[0,\pi]$, and then we continue on $[-\pi,0]$ as an even or odd function according to $k$.

We take a function equal to one on the interval $[k\pi/N, (k+1)\pi/N]$ and equal to zero outside this interval.
Then we smooth it near the ends $k\pi/N$ and $(k+1)\pi/N$, namely, on the intervals $k\pi/N -\pi/3N \le \w < k\pi/N+\pi/3N$ and $(k+1)\pi/N -\pi/3N\le \w <(k+1)\pi/N +\pi/3N$.
Then the following expressions for the smoothed functions $H^k_+(\w)$, $k = 1, 2, \dots, N-2$, are obtained for $\w \ge 0$:
\begin{equation}\label{eq-10}
H^k_+(\w)=
\left\{
  \begin{matrix}
    0, & \w < \frac{k\pi}{N} -\frac{\pi}{3N}\ {\rm or}\ \w\ge \frac{(k+1)\pi}{N} +\frac{\pi}{3N}, \\
    \sin\left(\frac \pi2 \nu\left(\frac{3N}{2\pi} \w -\frac{3k-1}{2}\right)\right), & \frac{k\pi}{N} -\frac{\pi}{3N} \le \w < \frac{k\pi}{N} +\frac{\pi}{3N}, \\
    1, & \frac{k\pi}{N} +\frac{\pi}{3N} \le \w < \frac{(k+1)\pi}{N} -\frac{\pi}{3N}, \\
    \cos\left(\frac \pi2 \nu\left(\frac{3N}{2\pi} \w -\frac{3(k+1)-1}{2}\right)\right), & \frac{(k+1)\pi}{N} -\frac{\pi}{3N} \le \w < \frac{(k+1)\pi}{N} +\frac{\pi}{3N}, \\
  \end{matrix}
\right .
\end{equation}

The last frequency function $H^{N-1}_+(\w)$ on $[0,\pi]$ is given by:
\begin{equation}\label{eq-11}
H^{N-1}_+(\w)=
\left\{
  \begin{matrix}
    \sin\left(\frac \pi2 \nu\left(\frac{3N}{2\pi} \w -\frac{3(N-1)-1}{2}\right)\right), & \frac{(N-1)\pi}{N} -\frac{\pi}{3N} \le \w < \frac{(N-1)\pi}{N} +\frac{\pi}{3N}, \\
    1, & \frac{(N-1)\pi}{N} +\frac{\pi}{3N} \le \w \le \pi, \\
    0, & \rm{for\ others}\ \w\in[0,\pi], \\
  \end{matrix}
\right .
\end{equation}

On the whole interval $[-\pi,\pi]$ of the function $H^k_+(\w)$ one can extend both an even and an odd function.
Next, we specify how to do this from the requirement that matrix (\ref{eq-7}) be unitary.
Outside the interval $[-\pi,\pi]$, the functions $H^k(\w)$ continue as $2\pi$-periodic functions.

Let us show that the above functions $H^0(\w)$, $H^k(\w)$, and (11) have the property that the rows of matrix (\ref{eq-7}) have unit norm.

The equality $|H^0(\w)|^2+ |H^0\left(\w-2\pi/N\right)|^2+\dots+ |H^0\left(\w-2\pi(N-1)/N\right)|^2=1$ is sufficient to show only on the interval $[\pi/N-\pi/3N,\pi/N+\pi/3N]$.
This follows from the fact that the smoothing of all functions at the discontinuity points is carried out equally and symmetrically.
In addition, in the reduced sum, the carriers of neighboring functions intersect exactly along the intervals where smoothing was performed. On the interval $[\pi/N-\pi/3N,\pi/N+\pi/3N]$ there are only two nonzero terms in the sum, $H^0(\w)$ and $H^0\left(\w-2\pi/N\right)$.
Then we obtain using the equality $\nu\left(1-x\right)=1-\nu\left(x\right)$:
$$
|H^0(\omega)|^2+|H^0(\omega-2\pi/N)|^2=
$$
$$
=\cos^2\left(\frac \pi2 \nu\left(\frac{3N}{2\pi} \w -1\right)\right)+
\cos^2\left(\frac \pi2 \nu\left(\frac{3N}{2\pi} \left|\w -\frac{2\pi}{N}\right|\right)\right)=
$$
$$
=\cos^2\left(\frac \pi2 \nu\left(\frac{3N}{2\pi} \w -1\right)\right)+
\cos^2\left(\frac \pi2 \nu\left(1-\left(\frac{3N}{2\pi}\w -1\right)\right)\right)=
$$
$$
=\cos^2\left(\frac \pi2 \nu\left(\frac{3N}{2\pi} \w -1\right)\right)+
\sin^2\left(\frac \pi2 \nu\left(\frac{3N}{2\pi}\w -1\right)\right)=1.
$$

Quite similarly, the normalization of the remaining rows for the functions $H^k(\w)$ is shown, both for the even continuation of $H^k_+(\w)$ and for the odd one.

Let us ensure the Hermitian orthogonality of the rows of the matrix (\ref{eq-7}).
Since the support of the function $H^k(\w)$ has a non-empty intersection with the supports of only the adjacent functions $H^{k-1}(\w)$ and $H^{k+1}(\w)$, it is sufficient to ensure that only adjacent rows are orthogonal.
To do this, it suffices to consider frequency functions $H^k(\w)$ with even numbers $k$ as even functions, and with odd $k$ consider odd functions $H^k(\w)$ on the interval $[-\pi,\pi]$ based on the functions $H^k_+(\w)$ defined on the interval $[0,\pi]$.

Thus, the desired frequency functions of wavelets $\psi^1(x), \dots , \psi^{N-1}(x)$ have the form: for even values of $k$, the function $H^k(\w)$ is even and on the interval $[0,\pi]$ is defined by formula (\ref{eq-10}), and for odd values of $k$ the function $H^k(\w)$ is odd and on the interval $[0,\pi]$ is defined by formula (\ref{eq-10}). Accordingly, the last frequency function $H^{N-1}(\w)$ is determined depending on the parity of the number $N-1$.

The scaling filter $\{h_n^0\}$ and the wavelet filters $\{h_n^k\}$, $k=1,2,\dots, N-1$, are found by Fourier series expansion of the frequency functions $H^0(\w)$ and $H^k(\w)$.

Wavelets $\psi^k(x)$ are determined by their filters $\{h_n^k\}$ by the formula (\ref{eq-4}). Since the Fourier transforms of the wavelets $\widehat{\varphi}(\w)$ and $\widehat{\psi^k}(\w)$ have compact supports, the functions $\varphi(x)$ and $\psi^1(x), \dots , \psi^{N-1}(x)$ are infinitely differentiable.

\section{Multiple scaling factors}
Note that in the case of an even $N>2$, the last frequency function $H^{N-1}(\w)$ defined above at the points $(2n+1)\pi$ has a discontinuity: to the left of $(2n+1)\pi$ it takes the value $1$, and to the right it takes the value $-1$.
As a result, the filter $\{h_n^{N-1}\}$ has an infinite impulse response and a weak decrease at infinity, which makes it difficult to use such wavelets in the case of an even $N>2$.
This defect can be overcome by using multiple scale factors and classical Meyer wavelets with a scale factor of 2.

Let $\varphi_H(x)$ and $\psi_H^1(x), \psi_H^2(x),\dots, \psi_H^{N-1}(x)$ be the scaling function and wavelets with a scaling factor $N$.
The corresponding filters and scaling relations in the frequency area:
$$
H^0(\omega)= \frac {1}{\sqrt{N}}\sum_{n\in \Z} h_n^0\ e^{-in\omega}, \quad H^k(\omega)= \frac {1}{\sqrt{N}}\sum_{n\in \Z}h_n^k\ e^{-in\omega}, \quad k =1,2,\dots ,N-1 .
$$
$$
\widehat{\varphi}_H(\omega) = H^0\left(\frac{\omega}{N}\right)\widehat{\varphi}_H\left(\frac{\omega}{N}\right),\quad \widehat{\psi_H^k}(\omega)=H^k\left(\frac{\omega}{N}\right)\widehat{\varphi}_H\left(\frac{\omega}{N}\right), \quad k=1,2,\ldots,N-1.
$$

Let $\varphi_G(x)$ and $\psi_G^1(x), \psi_G^2(x),\dots, \psi_G^{M-1}(x)$ be the scaling function and wavelets with a scaling factor $M$.
The corresponding filters and scaling relations in the frequency area:
$$
G^0(\omega)= \frac {1}{\sqrt{M}}\sum_{n\in \Z} g_n^0\ e^{-in\omega}, \quad G^l(\omega)= \frac {1}{\sqrt{M}}\sum_{n\in \Z}g_n^l\ e^{-in\omega}, \quad l =1,2,\dots ,M-1 .
$$
$$
\widehat{\varphi}_H(\omega) = G^0\left(\frac{\omega}{M}\right)\widehat{\varphi}_G\left(\frac{\omega}{M}\right),\quad \widehat{\psi_G^l}(\omega)=G^l\left(\frac{\omega}{M}\right)\widehat{\varphi}_M\left(\frac{\omega}{M}\right), \quad l=1,2,\ldots,M-1.
$$

\subsection{Wavelet decomposition with coefficient N}
Consider the question of how to obtain the scaling function and wavelets $\psi^{kl}(x)$, $k=0,1,2,\dots,N-1$, $l=0,1,2,\dots,M-1$ with the scaling factor $MN$ from the given $N$- and $M$-wavelets.
It is convenient to solve the problem at the level of formal power series \cite{Sm-14}.
Let $X(z)=\sum_n\, x_n z^n$ be the power series corresponding to the signal $\{x_n\}$, $z\in \CC$.
Let's first make a wavelet decomposition of the signal using wavelets with a scaling factor $N$, and then, to what happens, apply the wavelet decomposition with a factor $M$.

At the level of power series, wavelet decomposition (taking into account $N$-decimation) is performed according to the following scheme \cite{Sm-14}:
$$
X\to\{A^0,A^1,A^2,\dots,A^{N-1}\},
$$
where the components $A^k(\w)$ are also power series and are determined from the formulas:
\begin{equation}\label{eq-12}
A^k(\w)=A^k\left(z^N\right)=
\frac{1}{N}\sum_{s=0}^{N-1}{H^k\left(\rho^sz\right)X\left(\rho^sz\right)},\quad k=0,1,\ldots,N-1,			 \end{equation}
where $\rho=e^{-\frac{i2\pi}{N}}$ and $w=z^N$.
The power series coefficients of the $A^k(\w)$ are the coefficients $\{a_m^k\}$ of the wavelet expansion of the signal $\{x_n\}$.

\subsection{Wavelet decomposition with coefficient M}
Now, for each $A^k(w)$, we apply a wavelet decomposition with a scaling factor $M$.
At the level of power series, this wavelet decomposition is performed according to the following scheme:
$$
A^k(w)\to\{A^{k0}(u), A^{k1}(u),\dots, A^{k,M-1}(u)\},\quad k=0,1,2,\dots, N-1,
$$
where the components $A^{kl}(u)=A^{kl}(w^M)$ are determined from the formulas:
$$
A^{kl}\left(u\right)=A^{kl}\left(w^M\right)=\frac{1}{M}\sum_{p=0}^{M-1}{G^l\left(\tau^pw\right)A^k \left(\tau^pw\right)},\quad l=0,1,\ldots,M-1,
$$
where $\tau=e^{-i2\pi/M}$ and $u=w^N$.
As a result, we obtain a decomposition of the signal $X(z)$ into a matrix of formal power series $A^{kl}(u)$:
$$
X(z)\rightarrow\{A^0(w),A^1(w),\dots,A^{N-1}(w)\}\rightarrow
\left(\begin{matrix}
A^{00}(u)&A^{10}(u)&\dots&A^{N-1,0}(u)\\
A^{01}(u)&A^{11}(u)&\dots&A^{N-1,1}(u)\\
\dots&\dots&\dots&\dots\\
A^{0,M-1}(u)&A^{1,M-1}(u)&\dots&A^{N-1,M-1}(u)\\
\end{matrix}\right).
$$

To find $A^{kl}(u)$ we need expressions of the form $A^k(\tau^p w)$. If $w=z^N$, then $\tau^p w=\left(\tau^{p/N} z\right)^N$.
Then
$$
A^k\left(\tau^pw\right)=A^k\left(\tau^pz^n\right)= A^k\left(\left(\tau^\frac{p}{N}z\right)^N\right)= \frac{1}{N}\sum_{s=0}^{N-1}{H^k\left(\rho^s\tau^\frac{p}{N}z\right)X\left(\rho^s\tau^\frac{p}{N}z\right)}, \ k=0,1,\ldots,N-1.
$$

Let us calculate $A^{kl}(w^M)$. Let $\sigma=e^{-i2\pi/MN}$, then $\tau=\sigma^N$ and $\rho=\sigma^M$. Then:
$$
A^{kl}\left(w^M\right) =\frac{1}{M}\sum_{p=0}^{M-1}{G^l\left(\tau^pw\right)A^k\left(\tau^pw\right)} =\frac{1}{M}\frac{1}{N}\sum_{p=0}^{M-1}{G^l\left(\tau^pw\right)\sum_{s=0}^{N-1} {H^k\left(\rho^s\tau^\frac{p}{N}z\right)X\left(\rho^s\tau^\frac{p}{N}z\right)}}=
$$
$$
=\frac{1}{MN}\sum_{p=0}^{M-1}{G^l\left(\left(\tau^\frac{p}{N}z\right)^N\right)\sum_{s=0}^{N-1} {H^k\left(\rho^s\sigma^pz\right)X\left(\rho^s\sigma^pz\right)}} =
$$
$$
=\frac{1}{MN}\sum_{p=0}^{M-1}{G^l\left(\sigma^{pN}z^N\right)\,\sum_{s=0}^{N-1} {H^k\left(\sigma^{Ms+p}z\right)X\left(\sigma^{Ms+p}z\right)}}=
$$
\centerline{Denote $Ms+p=q$, $q=0,1,\dots,MN-1$. Then $p=q-Ms$. We get:}
$$
G^l(\sigma^{pN}z^N)=G^l(\sigma^{(q-Ms)N}z^N)=G^l(\sigma^{qN-MNs}z^N) =[\sigma^{MN}=1]=G^l(\sigma^{qN}z^N).
$$
We continue the calculations:
$$
=\frac{1}{MN}\sum_{q=0}^{MN-1}{G^l\left(\sigma^{qN}z^N\right)H^k\left(\sigma^qz\right) X\left(\sigma^qz\right)}= A^{kl}\left(w^M\right)
$$

We get the final formula:
$$
A^{kl}\left(w^M\right) =\frac{1}{MN}\sum_{q=0}^{MN-1}{G^l\left((\sigma^qz)^N\right)H^k\left(\sigma^qz\right) X\left(\sigma^qz\right)},
$$
where $\sigma=e^{-i2\pi/MN}$.

Now, as the frequency functions of the wavelets $\psi^{kl}(x)$, $k=0,1,\dots,N-1$, $l=0,1,\dots,M-1$, we take the functions
$$
H^{kl}(z)=G^l(z^N)H^k(z).
$$
Then,
$$
A^{kl}\left(z^{MN}\right)=\frac{1}{MN}\sum_{q=0}^{MN-1}{H^{kl}\left(\sigma^qz\right) X\left(\sigma^qz\right)}
$$
which fully corresponds to the general scheme (\ref{eq-12}) of the wavelet decomposition of the signal at the level of formal power series with the scaling factor $MN$.

We consider that $z=e^{-i\w}$ and we set $H^{kl}(\w)=H^{kl}(e^{-i\w})$. Thus, we have obtained the following wavelet frequency functions:
$$
H^{kl}(\w)=G^l(N\w)H^k(\w).
$$

Let us find the scaling filters of wavelets of the $\psi^{kl}(x)$.
$$
G^l\left(N\omega\right)H^k\left(\omega\right)=\frac{1}{\sqrt M}\sum_{m\in \Z}{g_m^le^{-iNm\omega}}\frac{1}{\sqrt N}\sum_{p\in \Z}{h_p^ke^{-ip\omega}}=\frac{1}{\sqrt{MN}}\sum_{p\in \Z}{\sum_{m\in \Z}{g_m^lh_p^ke^{-i\left(Nm+p\right)\omega}}=}
$$
\centerline{[denote $n=p+mN$, $p=n-mN$]}
$$
=\frac{1}{\sqrt{MN}}\sum_{n\in \Z}{\sum_{m\in \Z}{g_m^lh_{n-mN}^ke^{-in\omega}}=}\frac{1}{\sqrt{MN}}\sum_{n\in \Z}{h_n^{kl}e^{-in\omega}=}H^{kl}\left(\omega\right)
$$

Thus, we obtain the scaling filters of the wavelets $\psi^{kl}(x)$:
$$
h_n^{kl}=\sum_{m\in \Z}{g_m^lh_{n-mN}^k},\quad k=0,1,\ldots,N-1,\quad l=0,1,\ldots,M-1 .
$$

\begin{theorem} \label{Th2}
If $\varphi_H(x)$ and $\varphi_G(x)$ are $N$- and $M$-Meyer scaling functions $\varphi(x)$ with frequency functions $H^0(\w)$ and $G^0(\w)$, respectively, then the scaling function $\psi^{00}(x)$ corresponding to the scaling filter $H^{00}(\w)=G^0(N\w)H^0(\w)$, is also the Meyer function $\varphi(x)$.
\end{theorem}

\begin{proof}
It suffices to show that the Meyer function $\varphi(x)$ satisfies the scaling relation with coefficients $h_n^{00}=\sum_{m\in \Z}{g_m^0h_{n-mN}^0}$:
$$
\psi^{00}\left(x\right)=\sqrt{NM}\sum_{n\in \Z}\sum_{m\in \Z}{g_m^0h_{n-mN}^0\varphi\left(MNx-n\right)}=
$$
\centerline{[denote $k = n-mN, n=k+mN$]}
$$
=\sqrt{NM}\sum_{k\in \Z}\sum_{m\in \Z}{g_m^0h_k^0\varphi\left(MNx-k-mN\right)}=
\sqrt{NM}\sum_{m\in \Z}g_m^0 \sum_{k\in \Z}{h_k^0 \varphi\left(N\left(Mx-m\right)-k\right)}=
$$
$$
=\sqrt M\sum_{m\in \Z}{g_m^0\varphi\left(Mx-m\right)}=\varphi\left(x\right)
$$
\end{proof}

In addition, we have:
$$
\widehat{\psi^{00}}\left(\omega\right)= H^{00}\left(\frac{\omega}{MN}\right)\hat{\varphi}\left(\frac{\omega}{MN}\right) =G^0\left(\frac{\omega}{M}\right)H^0\left(\frac{\omega}{MN}\right) \hat{\varphi}\left(\frac{\omega}{MN}\right)=G^0\left(\frac{\omega}{M}\right) \hat{\varphi}\left(\frac{\omega}{M}\right)=\hat{\varphi}\left(\omega\right)
$$

It follows from the last theorem that the wavelets $\widehat{\psi^{kl}}(\omega)$ can be found from the relations
$$
\widehat{\psi^{kl}}(\omega)=H^{kl}\left(\frac{\omega}{MN}\right)\hat{\varphi} \left(\frac{\omega}{MN}\right) =G^l\left(\frac{\omega}{M}\right)H^k\left(\frac{\omega}{MN}\right) \widehat{\varphi}\left(\frac{\omega}{MN}\right)
$$
where the function $\widehat{\varphi}(\w)$ is defined by formula (\ref{eq-8}), and the frequency functions $G^l(\w)$ and $H^k(\w)$ are assumed to be known.

\subsection{Example. Meyer wavelets with scaling factor 6}
Consider the construction of Meyer wavelets with a multiple scaling factor $MN$ when $M=3$ and $N=2$. The scaling function will be the same $\varphi(x)$ indicated in formula (\ref{eq-8}).
We recall the scheme for constructing Meyer wavelets for $N=2$ \cite{Daub}, \cite{Sm-14}.

Let $H^0\left(\omega\right)=\frac{1}{\sqrt 2}\sum_{n\in Z}{h_n^0e^{-in\omega}}$ be the frequency function of the Meyer wavelet $\varphi(x)$ for $N=2$,
$$
H^0(\w)=
\left\{
  \begin{matrix}
    1, &\w\in [-\frac{2\pi}{6},\frac{2\pi}{6}], \\
    \cos\left(\frac \pi2 \nu\left(\frac{6}{2\pi}|\w|-1\right)\right), & \frac{2\pi}{6} \le|\w|\le \frac{4\pi}{6}, \\
    0, & \rm{for\ others}\ \w\in[-\pi,\pi] ,\\
  \end{matrix}
\right .
$$

Then the wavelet $\psi(x)$ is found from the formula $\widehat{\psi}\left(\omega\right)= e^\frac{i\omega}{2}\overline{H^0\left(\frac{\omega}{2}+ \pi\right)}\widehat{\varphi}\left(\frac{\omega}{2}\right).$
The corresponding frequency function has the view $H^1\left(\omega\right)=e^{i\omega}\overline{H^0 \left(\omega+\pi\right)}$ and its filter $\{h_n^1\}$ is found by the formula $h_n^1=(-1)^{n+1}h_{-n-1}^0$.

In the case of $M=3$, we take the same scaling function $\varphi(x)$ with the scaling relation $\varphi(x)=\sqrt3\sum_{n\in \Z}{g_n^0\varphi(3x-n)}$, where the scaling filter $\{g_n^0\}$ is found from the Fourier series expansion of the function $G^0(\w)$ defined by formula (\ref{eq-9}) for $N=3$:
$$
G^0(\w)=
\left\{
  \begin{matrix}
    1, &\w\in [-\frac{2\pi}{9},\frac{2\pi}{9}], \\
    \cos\left(\frac \pi2 \nu\left(\frac{9}{2\pi}|\w|-1\right)\right), & \frac{2\pi}{9} \le|\w|\le \frac{4\pi}{9}, \\
    0, & \rm{for\ others}\ \w\in[-\pi,\pi] ,\\
  \end{matrix}
\right .
$$

Wavelet filters $\{g_n^1\}$, $\{g_n^2\}$ are found by expanding the frequency functions $H^1(\w)$ and $H^2(\w)$ defined by formulas (\ref{eq-10}) and (\ref{eq-11}) into a Fourier series.

Following the scheme outlined in the previous paragraph, we define the frequency functions of the wavelets $\psi^{kl}(x)$, $k=0,1$,\, $l=0,1,2$ by the formula $H^{kl}(\w)=G^l(2\w)H^k(\w)$. Here $\psi^{00}(x)=\varphi(x)$ is the Meyer scaling function.
Thus, we obtain a matrix of wavelet frequency functions $\psi^{kl}(x)$ with a scaling factor of 6
$$
G\left(\omega\right)=
\left(\begin{matrix}H^{00}\left(\omega\right)&H^{01}\left(\omega\right)&H^{02 }\left(\omega\right)\\H^{10}\left(\omega\right)&H^{11}\left(\omega\right)&H^{12}\left(\omega\right)\\
\end{matrix}\right)\ =\left(\begin{matrix}G^0\left(2\omega\right)H^0\left(\omega\right)&G^1\left(2\omega\right)H^0 \left(\omega\right)&G^2\left(2\omega\right)H^0\left(\omega\right)\\
G^0\left(2\omega\right) H^1\left(\omega\right)& G^1\left(2\omega\right)H^1\left(\omega\right)&G^2\left(2\omega\right)H^1 \left(\omega\right)\\
\end{matrix}\right)
$$
and wavelet filters:
$$
h_n^{kl}=\sum_{m\in Z}{g_m^lh_{n-2m}^k},\quad k=0,1,\quad l=0,1,2.
$$

Let us present graphs of frequency functions $H^{kl}(\w)= G^l(N\w)H^k(\w)$ (Fig.1 and Fig.2).

\begin{figure}[h]
\center{\includegraphics[width=0.6\linewidth]{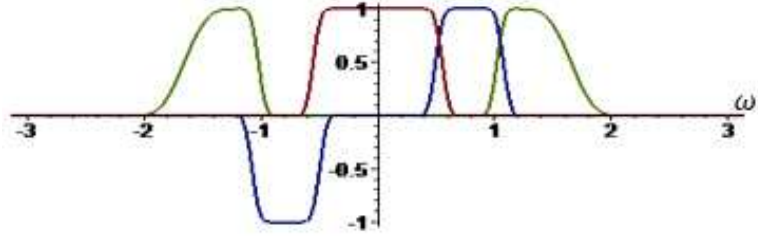}}
\caption{Plots of frequency functions $H^{00}(\w)$ (red), $H^{01}(\w)$ (blue) and $H^{02}(\w)$ (green).}
\label{Fig_1}
\end{figure}

\begin{figure}[h]
\center{\includegraphics[width=0.6\linewidth]{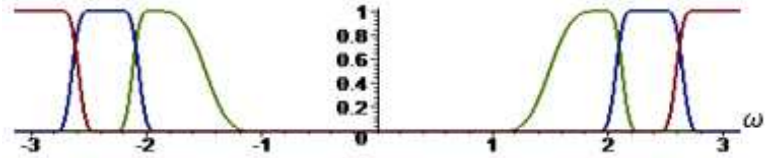}}
\caption{Plots of moduli of frequency functions $H^{10}(\w)$ (red), $H^{11}(\w)$ (blue) and $H^{12}(\w)$ (green).}
\label{Fig_2}
\end{figure}

\end{document}